\documentclass[12pt]{amsart}
\usepackage{bbm}
\numberwithin{equation}{section}
\makeatletter

\newcommand{\Rmnum}[1]{\expandafter\@slowromancap\romannumeral #1@}
\makeatother
\def\eps{\varepsilon}

\def\R{\mathbb R}
\def\D{\mathbb D}
\def\C{\mathbb C}

\def\N{\mathbb N}
\def\Z{\mathbb Z}
\def\im{\operatorname{Im}}
\def\re{\operatorname{Re}}

\def\dim{\operatorname{dim}}
\def\dimH{\operatorname{dim_H}}
\def\dimP{\operatorname{dim_P}}
\def\dimUB{\operatorname{\overline{dim}_B}}

\def\diam{\operatorname{diam}}
\def\area{\operatorname{area}}

\def\ann{\operatorname{ann}}

\def\length{\operatorname{length}}

\newtheorem{lemma}{Lemma}[section]
\newtheorem{theorem}{Theorem}[section]
\newtheorem{cor}{Corollary}[section]
\newtheorem*{theorema}{Theorem A}
\newtheorem*{theoremb}{Theorem B}

\newtheorem{proposition}{Proposition}[section]
\theoremstyle{definition}

\theoremstyle{remark}

\newtheorem{remark}{Remark}[section]

\title[Packing dimension of
Julia sets]{On the packing dimension of the Julia set and the escaping
set of an entire function}
\author{Walter Bergweiler}\thanks{Supported
by a Chinese Academy of Sciences Visiting
Professorship for Senior International Scientists, Grant
No.\ 2010 TIJ10. Also supported by 
the Deutsche Forschungsgemeinschaft, Grant Be 1508/7-1,
the EU Research Training Network CODY
and 
 the ESF Networking Programme HCAA}
\email{bergweiler@math.uni-kiel.de}
\address{Mathematisches Seminar,
Christian--Albrechts--Universit\"at zu Kiel,
Lude\-wig--Meyn--Str.~4,
D--24098 Kiel,
Germany}
\begin{document}
\begin{abstract}
Let $f$ be a  transcendental entire function.
 We give conditions which imply that the Julia set and the escaping set of $f$ have
 packing dimension~$2$.
 For example,
 this holds if there exists a positive constant $c$
 less than $1$ such that the minimum modulus $L(r,f)$ and the
 maximum modulus $M(r,f)$ satisfy $\log L(r,f)\leq c \log M(r,f)$
 for large~$r$.
 The conditions are also satisfied if $\log M(2 r,f)\geq d \log
 M(r,f)$ for some constant $d$ greater than $1$ and all large~$r$.
\end{abstract} \maketitle
\section{Introduction and results}\label{intro}
The Fatou set $F(f)$ of an entire function $f$ is defined as
the set of all $z\in\C$ where the iterates $f^n$ of $f$
form a normal
family. The  Julia set is the complement of $F(f)$
and denoted by $J(f)$. The escaping set $I(f)$ is the set of
all $z\in\C$ for which $f^n(z)\to\infty$ as $n\to\infty$.
We note that $J(f)=\partial I(f)$ by a result of Eremenko~\cite{Eremenko89}.
For an introduction to the iteration theory of transcendental
entire functions we refer to~\cite{Bergweiler93}.

Considerable attention has been paid to the dimensions of Julia sets
of entire functions; see~\cite{Stallard08} for a survey,
as well as~\cite{Baranski08,BKZ,BK,BKS,BRS1,Rempe09,Rempe10,Schubert07}
for some recent results not covered there.
Many
results in this area are concerned with the Eremenko-Lyubich class
$B$ consisting of all transcendental
entire functions for which the set of critical
and finite asymptotic values is bounded. By a result of Eremenko and
Lyubich~\cite[Theorem~1]{Eremenko92} we have $I(f)\subset J(f)$ for
$f\in B$. For a function in the Eremenko-Lyubich class, a lower
bound for the dimension of the Julia set can thus be obtained from
such a bound for the escaping set.
This played a key role
 already in McMullen's seminal paper~\cite{McMullen87}, 
and it  has been used in many subsequent papers.

We denote the Hausdorff dimension, packing dimension
and  upper box dimension
of a subset $A$ of the complex
plane $\C$ by $\dimH A$, $\dimP A$ and $\dimUB A$,
respectively,
noting that the upper box dimension
is defined only for bounded sets~$A$.
We refer to the book by Falconer~\cite{Falconer}
for the definitions
and a thorough treatment
of these concepts. Here we only note that we always
have~\cite[p.~48]{Falconer}
\[
\dimH A\leq \dimP A\leq \dimUB A.
\]
The exceptional set $E(f)$ of a
transcendental entire function $f$ consists of all
points in $\C$ with finite backward orbit. It is an immediate
consequence of Picard's theorem that $E(f)$ contains at most one
point.  The following result is one part of a theorem of Rippon and
Stallard~\cite[Theorem~1.2]{Rippon05}.
\begin{theorema}
Let $f$ be a transcendental entire function, $A$ a backward
invariant subset of $J(f)$ and $U$ a bounded open subset
of $\C$ whose closure
does not intersect $E(f)$.
 Then $\dimUB(U\cap
A)=\dimP A$. In particular,  $\dimUB(U\cap J(f))=\dimP J(f)$.
\end{theorema}
For functions in the Eremenko-Lyubich class  they obtained the
following result~\cite[Theorem~1.1]{Rippon05}.
\begin{theoremb}
Let $f\in B$. Then $\dimP J(f)=2$.
\end{theoremb}
It follows from Theorem~A that Theorem~B is equivalent to the result
that $\dimUB(U\cap J(f))=2$ for some bounded open set~$U$
satisfying $\overline{U}\cap  E(f)=\emptyset$. In order
to show this, Rippon and Stallard actually proved that $\dimUB(U\cap
I(f))=2$ for such a set $U$
and then used the result of Eremenko and Lyubich quoted
above.

The main tool used by Eremenko and Lyubich
to prove this result is a logarithmic change of variable. 
This method shows in particular
that if $f\in B$, then $f$ is bounded on a curve tending
to~$\infty$; see~\cite[p.~993]{Eremenko92}.
 The $\cos\pi\rho$-theorem (see~\cite[Chapter~5,
Theorem~3.4]{Goldberg08}
or~\cite[Section~6.2]{Hayman89})
now implies that
\begin{equation}\label{1a3}
\liminf_{r\to\infty}\frac{\log M(r,f)}{\sqrt{r}}>0,
\end{equation}
where
\[
M(r,f):=\max_{|z|=r}|f(z)|
\]
is the maximum modulus. (The observation that~\eqref{1a3} holds for
functions in $B$ seems to have been made first in~\cite[Proof of
Corollary~2]{Bergweiler95} and \cite[p.~1788]{Langley95};
 see also~\cite[Lemma~3.5]{Rippon05}.)

It follows from a result of Baker~\cite[Corollary to Theorem~3.1]{Baker84}
that all components of $F(f)$ are simply connected if $f$ is bounded
on a curve tending to $\infty$. In particular, if $f\in B$, then
$F(f)$ has no multiply connected components~\cite[Proposition~3]{Eremenko92}.

In view of these results
 the following theorem can be considered as a generalization
of Theorem~B.
\begin{theorem}\label{thm1}
Let $f$ be a transcendental entire function satisfying
\begin{equation}\label{1a}
\liminf_{r\to\infty}\frac{\log\log M(r,f)}{\log\log r}=\infty.
\end{equation}
If $F(f)$ has no multiply connected component,
then
\begin{equation}\label{1a4}
\dimP(I(f)\cap J(f))=2.
\end{equation}
\end{theorem}
Since multiply connected components of $F(f)$
are contained in $I(f)$, we see that
$I(f)$ has interior points if $F(f)$ has such a
component. We conclude that $\dim_P I(f)=2$
for all entire functions satisfying~\eqref{1a}.

The following result is an immediate consequence of Theorem~\ref{thm1}
and the results stated before it.
\begin{cor}\label{cor1}
Let $f$ be a transcendental entire
 function which is bounded on a
curve tending to~$\infty$. Then $\dimP(I(f)\cap J(f))=2$.
\end{cor}
More generally, we have the following result involving
the minimum modulus
\[
L(r,f):=\min_{|z|=r}|f(z)|.
\]
\begin{cor}\label{cor2}
Let $f$ be a transcendental entire function and suppose that
\begin{equation}\label{1x}
\limsup_{r\to\infty}\frac{\log L(r,f)}{\log M(r,f)}<1.
\end{equation}
Then $\dimP(I(f)\cap J(f))=2$.
\end{cor}
To deduce Corollary~\ref{cor2} from Theorem~\ref{thm1}
we note that the
$\cos\pi\rho$-theorem  yields that
\begin{equation}\label{1yy}
\liminf_{r\to\infty}\frac{\log\log M(r,f)}{\log r}>0
\end{equation}
under the hypothesis~\eqref{1x}. Clearly~\eqref{1a}
follows from~\eqref{1yy}.

Zheng~\cite[Corollary~1]{Zheng06}
proved, as a corollary to the main result of his paper,
 that if $F(f)$ has a
multiply connected component, then
\begin{equation}\label{1z}
\limsup_{r\to\infty}\frac{\log L(r,f)}{\log M(r,f)}>0.
\end{equation}
A slight extension of his argument shows that his main result
actually yields that
\begin{equation}\label{1y}
\limsup_{r\to\infty}\frac{\log L(r,f)}{\log M(r,f)}=1
\end{equation}
if $F(f)$ has a multiply connected component;
see~Proposition~\ref{zhengmod}. Thus the hypotheses of
Theorem~\ref{thm1} are satisfied if~\eqref{1x} holds.
  Actually the
condition~\eqref{1x} can be further relaxed; cf. Remark~\ref{rem2}.

 Zheng~\cite[Corollary~5]{Zheng06} also showed that the Fatou set
 of a transcendental entire function $f$ has no multiply connected
 component if
\begin{equation}\label{1zh}
\log M(2r,f)\geq d \log M(r,f)
\end{equation}
for some $d>1$ and all large~$r$. It is easy to see that~\eqref{1zh}
implies~\eqref{1yy} and hence~\eqref{1a}. Thus we obtain the
following corollary to Theorem~\ref{thm1}.
\begin{cor}\label{cor4}
Let $f$ be a transcendental entire function satisfying~\eqref{1zh}.
 Then $\dimP(I(f)\cap J(f))=2$.
\end{cor}
We mention that in Theorem B and in Theorem~\ref{thm1}, as well as
in the corollaries to Theorem~\ref{thm1}, the packing dimension
cannot be replaced by the Hausdorff dimension. In fact, it is shown
in~\cite[Corollary~1.4]{Rempe10} that there exists a function $f\in
B$ for which $\dimH I(f)=1$ and in~\cite{Stallard90} that for every
$\eps>0$ there exists a function $f\in B$ such that $\dimH
J(f)<1+\eps$, and
the functions considered in~\cite{Rempe10,Stallard90}
satisfy~\eqref{1zh} as well.
On the other hand, for every transcendental entire function $f$ the
Hausdorff dimension of $I(f)\cap J(f)$ is at least~$1$, since this
set contains continua; see~\cite[Theorem~5]{Rippon05a}
 and~\cite[Theorem~1.3]{Rippon10}.

Concerning the proof of Theorem~\ref{thm1}, we note that
in view of Theorem~A the conclusion of Theorem~\ref{thm1} is
equivalent to the statement that
\begin{equation}\label{1a2}
\dimUB(U\cap I(f)\cap J(f))=2
\end{equation}
for some bounded open set~$U$ satisfying $\overline{U}\cap E(f)=\emptyset$,
which in turn is equivalent to~\eqref{1a2} holding for all
bounded open sets~$U$.
We shall show in our proof that~\eqref{1a2} holds for some bounded
open set~$U$ 
whose closure does not intersect $E(f)$.

The main tools used in the proof
are certain estimates of the logarithmic derivative
and a version of the Ahlfors islands theorem. A similar
technique was used in~\cite{BK} where it was shown that
under a suitable regularity condition on the growth
of $f$ we even have
$\dimH(I(f)\cap J(f))=2$.

\section{Preliminary lemmas}\label{prelim}
We use the standard terminology of Nevanlinna theory and,
in particular, denote by $T(r,f)$ the Nevanlinna characteristic
of a meromorphic function~$f$;
see~\cite{Goldberg08,Hayman64}.
First we note that using
the well-known inequality
\begin{equation}\label{2a}
T(r,f)\leq \log^+ M(r,f)\leq \frac{R+r}{R-r} T(R,f)
\quad\text{for}\ 0<r<R
\end{equation}
we easily see that the growth condition~\eqref{1a}
is equivalent to
\begin{equation}\label{2b}
\liminf_{r\to\infty}\frac{\log T(r,f)}{\log\log r}=\infty.
\end{equation}

We shall need a number of lemmas and
begin with the following estimate of the logarithmic
derivative~\cite[p.~88]{Goldberg08}.
\begin{lemma}\label{goldberg}
Let $f$ be an entire function satisfying $f(0)=1$.
Then
\[
\left| \frac{f'(z)}{f(z)}\right|
\leq \frac{4s}{(s-|z|)^2} T(s,f)+
\sum_{|z_j|\leq s} \frac{2}{|z-z_j|},
\]
for $s>|z|$,
where $(z_j)$ is the sequence of zeros of~$f$.
\end{lemma}

In order to estimate the sum on the right hand side
we shall use the following result due to Fuchs and
Macintyre~\cite{Fuchs40}. Here and in the following
we denote by $D(a,r)$ the open disk 
and by $\overline{D}(a,r)$ the closed disk
of radius $r$ around
a point~$a$.
\begin{lemma} \label{fuchs}
Let $z_1,z_2,\dots,z_n\in\C$ and let $H>0$.
\begin{itemize}
\item[(i)]
There exist $l\in\{1,2,\dots,n\}$,
$u_1,u_2,\dots,u_l\in\C$ and $s_1,s_2,\dots,s_l>0$ satisfying
\[
\sum_{k=1}^l s_k^2 \leq 4H^2
\]
such that
\[
\sum_{k=1}^m\frac{1}{|z-z_k|}\leq  \frac{2n}{H}
\quad\text{for}\
z
\notin \bigcup_{k=1}^l D(u_k,s_k).
\]
\item[(ii)]
There exist
$m\in\{1,2,\dots,n\}$,
$v_1,v_2,\dots,v_m\in\C$ and $t_1,t_2,\dots,t_m>0$ satisfying
\[
\sum_{k=1}^m t_k \leq 2H
\]
such that
\[
\sum_{k=1}^n\frac{1}{|z-z_k|}\leq  \frac{n(1+\log n)}{H}
\quad\text{for}\
z
\notin \bigcup_{k=1}^m D(v_k,t_k).
\]
\end{itemize}
\end{lemma}
Actually Fuchs and Macintyre write $An/H$ instead of $2n/H$ in
part~(i), with a constant~$A$, but their argument shows that one can
take $A=2$.
We note that the term $\log n$ in (ii) cannot be omitted; 
cf.~\cite{AndEid}.

We will also require the following version of the Borel-Nevanlinna
growth lemma;
see~\cite[Chapter~3, Theorem~1.2]{Goldberg08}
or~\cite[Section~3.3]{Cherry01}.
 Here
a measurable subset $E$ of $(0,\infty)$ is said to be of finite
logarithmic measure if
\[
\int_E\frac{dt}{t}<\infty.
\]
\begin{lemma}\label{growth}
Let $F:[r_0,\infty)\to [t_0,\infty)$
and $\varphi:[t_0,\infty)\to(0,\infty)$
be non-decreasing functions,
with $r_0,t_0>0$.
Suppose that
\[
\int_{t_0}^\infty\frac{dt}{\varphi(t)}<\infty.
\]
Then there exists a set $E\subset [r_0,\infty)$
 of finite logarithmic measure such that
\[
F\left(r\left(1+\frac{1}{\varphi(F(r))}\right)\right) \leq F(r)+1
\quad \text{for}\ r\notin E.
\]
\end{lemma}
We shall also need a result from the Ahlfors theory of covering
surfaces; cf.~\cite[Theorem~6.2]{Hayman64} or~\cite{Bergweiler98}.
Here $\D:=D(0,1)$ is the unit disk.
\begin{lemma}\label{ahlfors}
Let $D_1,D_2,D_3$ be Jordan domains with pairwise disjoint closures.
Then there exists $\mu>0$ such that if
$h:\D\to\C$ is a holomorphic function satisfying
\[
\frac{|h'(0)|}{1+|h(0)|^2}\geq \mu,
\]
then $\D$ has a subdomain which is mapped bijectively
onto one of the domains~$D_\nu$ by~$h$.
\end{lemma}
Finally, we shall repeatedly   use the following result known as 
the 
 Koebe distortion theorem and
the Koebe one quarter  theorem.

\begin{lemma}\label{koebe}
Let $g:D(a,r)\to \C$ be univalent, $0< \lambda <1$ and $z\in 
\overline{D}(a,\lambda r)$. Then
\[
\frac{1}{(1+\lambda)^2}|g'(a)|\leq \frac{|g(z)-g(a)|}{|z-a|}
\leq \frac{1}{(1-\lambda)^2}|g'(a)|
\]
and
\[
\frac{1-\lambda}{(1+\lambda)^3}|g'(a)|\leq |g'(z)|\leq
\frac{1+\lambda}{(1-\lambda)^3}|g'(a)|.
\]
Moreover,
\[
g(D(a,r))\supset D\left(g(a),\frac{1}{4}|g'(a)|r\right).
\]
\end{lemma}
Usually Koebe's theorems are stated only for the special case 
that $a=0$, $r=1$, $g(0)=0$ and $g'(0)=1$, but the above
result easily follows from this special case.

\section{Proof of Theorem~\ref{thm1}}
Without loss of generality we may assume that $f(0)=1$.
We denote by $n(r,0)$ the number of zeros of $f$ in the
closed disk of radius $r$ around $0$ and put
\[
N(r,0):=\int_0^r\frac{n(t,0)}{t}dt.
\]
By Nevanlinna's first fundamental theorem we have
\begin{equation}\label{2d}
T(s,f)
\geq N(s,0)\geq\int_r^s\frac{n(t,0)}{t}dt\geq n(r,0) \int_r^s\frac{dt}{t}
= n(r,0) \log\left(\frac{s}{r}\right)
\end{equation}
for $s>r>0$. Applying Lemma~\ref{growth} with $\varphi(x)=x^2/6$ and
$F(r)=\log T(r,f)$ we obtain a set $E$ of finite logarithmic measure
such that
\begin{equation}\label{3a}
T\left(r\left(1+\frac{6}{[\log T(r,f)]^2}\right),f\right) \leq e T(r,f)
\quad \text{for}\ r\notin E.
\end{equation}
For $m\in\N$ we shall use the abbreviation
\[
R_m(r):=r\left(1+\frac{m}{[\log T(r,f)]^2}\right)
\]
so that~\eqref{3a} takes the form
\begin{equation}\label{3b}
T\left(R_6(r),f\right)\leq e T(r,f) \quad \text{for}\ r\notin E.
\end{equation}
Using~\eqref{2a} we find that if $r\notin E$ is sufficiently large,
then
\begin{equation}\label{3b1}
\log M(r,f)\leq \frac{R_6(r)+r}{R_6(r)-r}T(R_6(r),f)
\leq T(r,f)[\log T(r,f)]^2.
\end{equation}

For measurable $X\subset\R$ and $Y\subset \C$
we denote by $\length X$ the $1$-dimensional Lebesgue measure 
of $X$ and by $\area Y$ the $2$-dimensional Lebesgue measure
of~$Y$.
\begin{lemma}\label{lemma31}
For sufficiently large $r\notin E$  there exists
a closed subset $F_r$ of $\left[R_1(r),R_3(r)\right]$
with
\begin{equation}\label{3c}
\length F_r\geq \frac{r}{[\log T(r,f)]^2}
\end{equation}
such that
\begin{equation}\label{3d}
\left|\frac{f'(z)}{f(z)}\right|
\leq \frac{T(r,f) [\log T(r,f)]^7}{r}
\quad\text{for}\ |z|\in F_r.
\end{equation}
Moreover, for $\eta>0$ there exist
an integer $l\in\left\{1,2,\dots, n\left(R_5(r),0\right)\right\}$,
points $u_1,u_2,\dots,u_l\in\C$ and $s_1,s_2,\dots,s_l>0$ satisfying
\begin{equation}\label{3d1}
\sum_{k=1}^l s_k^2 \leq \frac{r^2}{T(r,f)^\eta}
\end{equation}
such that
\begin{equation}\label{3e}
\left|\frac{f'(z)}{f(z)}\right|
\leq \frac{T(r,f)^{1+\eta}}{r}
\quad\text{if}\
r\leq |z|\leq R_4(r)
\ \text{and}\ z\notin \bigcup_{k=1}^l D(u_k,s_k).
\end{equation}
\end{lemma}
\begin{proof}
It follows from Lemma~\ref{goldberg} that if $z_1,z_2,\dots$ are
the zeros of $f$, then
\begin{equation}\label{3g}
\left| \frac{f'(z)}{f(z)}\right|
\leq \frac{4R_5(r)}{(R_5(r)-|z|)^2} T(R_5(r),f)+
\sum_{|z_j|\leq R_5(r)} \frac{2}{|z-z_j|}
\end{equation}
for $|z|< R_5(r)$. Using~\eqref{3b} and noting that $R_5(r)\leq 2r$
for large $r$ we see that if $|z|\leq R_4(r)$, then
\begin{equation}\label{3h}
\frac{4R_5(r)}{(R_5(r)-|z|)^2} T(R_5(r),f) \leq 8e
\frac{T(r,f)}{r}[\log T(r,f)]^4
\end{equation}
for large $r\notin E$.
To estimate the sum on the right hand side of~\eqref{3g} we apply
Lemma~\ref{fuchs}, part~(ii), with
\begin{equation}\label{3i}
H=\frac{r}{8[\log T(r,f)]^2}
\end{equation}
and conclude that
\begin{equation}\label{3j}
\sum_{|z_j|\leq R_5(r)} \frac{1}{|z-z_j|}
\leq \frac{n(R_5(r),0)(1+\log n(R_5(r),0))}{H}
\end{equation}
outside a union of disks whose sum of radii is at most $2H$. Let $P$
be the set of all $s>0$ such that $\{z\in\C:|z|=s\}$ intersects the
union of these disks. Then $\length P\leq 4H$. Thus
$F_r:=[R_1(r),R_3(r)]\setminus P$ satisfies~\eqref{3c},
and~\eqref{3j} holds for $|z|\in F_r$.

From~\eqref{2d} and~\eqref{3b} we can deduce that
\begin{equation}\label{3j1}
\begin{aligned}
 n(R_5(r),0)
&\leq
\frac{T(R_6(r),f)}{\log\left(R_6(r)/R_5(r)\right)}\\
&\leq 2 T(R_6(r),f) [\log T(r,f)]^2
\leq 2e T(r,f) [\log T(r,f)]^2
\end{aligned}
\end{equation}
for large $r\notin E$.
  Combining this with~\eqref{3i} and~\eqref{3j}
we obtain
\begin{equation}\label{3k}
\begin{aligned}
&\quad \sum_{|z_j|\leq R_5(r)} \frac{1}{|z-z_j|}\\
&\leq \frac{16 e T(r,f) [\log T(r,f)]^4\left(
1+\log\left(2e T(r,f) [\log T(r,f)]^2\right) \right)}{r}\\
&\leq \frac{T(r,f)[\log T(r,f)]^6}{r}
\end{aligned}
\end{equation}
if $|z|\in F_r$ and $r$ is large.
Now~\eqref{3d} follows from~\eqref{3g}, \eqref{3h} and~\eqref{3k}.

In order to prove~\eqref{3e} we use part~(i) of Lemma~\ref{fuchs}
with
\begin{equation}\label{3l}
H=\frac{r}{2 T(r,f)^{\eta/2}}.
\end{equation}
This yields
$l$ disks $D(u_k,s_k)$ satisfying~\eqref{3d1} such that
\begin{equation}\label{3l1}
\sum_{|z_j|\leq R_5(r)} \frac{1}{|z-z_j|}
\leq \frac{2 n(R_5(r),0)}{H}
\quad\text{for}\
z\notin \bigcup_{k=1}^l D(u_k,s_k),
\end{equation}
with $l\leq n(R_5(r),0)$.
Combining~\eqref{3j1}, \eqref{3l} and~\eqref{3l1} we find that
\[
\quad \sum_{|z_j|\leq R_5(r)} \frac{1}{|z-z_j|} \leq \frac{8e T(r,f)
[\log T(r,f)]^2 T(r,f)^{\eta/2}}{r}
\]
if $z\notin \bigcup_{k=1}^l D(u_k,s_k)$.
This, together with~\eqref{3g} and~\eqref{3h},
implies~\eqref{3e}.
\end{proof}
\begin{lemma}\label{lemma32}
Let $F_r$ be as in Lemma {\rm \ref{lemma31}} and let $\delta>0$. If
$r$ is sufficiently large, then for each $s\in F_r$ there exists a
closed subset $J_s$ of $[0,2\pi]$ with
\begin{equation}\label{3n1}
\length J_s \geq \frac{1}{T(r,f)^\delta}
\end{equation}
such that if $\theta\in J_s$, then
\begin{equation}\label{3n}
|f(se^{i\theta})|\geq \sqrt{M(r,f)}
\end{equation}
and
\begin{equation}\label{3o}
\left|\frac{f'(se^{i\theta})}{f(se^{i\theta})}\right|
\geq \frac{T(r,f)^{1-\delta}}{r}.
\end{equation}
\end{lemma}
\begin{proof}
First we consider the case that
\begin{equation}\label{p3h}
\log L(s,f)=\min_{|z|=s}\log|f(z)|\leq\frac12 \log M(s,f).
\end{equation}
Then there exists $z_1$ and $z_2$ with $|z_1|=|z_2|=s$
satisfying
\[
\log|f(z_1)|=\frac12 \log M(s,f)
\quad\text{and}\quad
\log|f(z_2)|=\log M(s,f)
\]
while
\[
\frac12 \log M(s,f)\leq\log|f(z)|\leq\log M(s,f)
\]
on one of the two arcs between $z_1$ and $z_2$.
With $z_1=se^{i\theta_1}$ and $z_2=se^{i\theta_2}$ we may assume
without loss of generality
that $\theta_1<\theta_2$ and
\[
\frac12 \log M(s,f)\leq\log|f(se^{i\theta})|\leq\log M(s,f)
\]
for $\theta_1\leq \theta\leq \theta_2$. Let $J_s$ be the set of all
$\theta\in[\theta_1,\theta_2]$ for which~\eqref{3o} holds.
Since~\eqref{3n} holds for all $\theta\in[\theta_1,\theta_2]$
and thus in particular for $\theta\in J_s$, we
only have to estimate the length of~$J_s$. By the choice of $z_1$
and $z_2$ we have
\begin{equation}\label{p3ww}
\begin{aligned}
\frac12 \log M(s,f)
&=\log|f(z_2)|-\log|f(z_1)|
=\re\left(\int_{z_1}^{z_2}
\frac{f'(z)}{f(z)} dz \right)\\
&\leq
\int_{z_1}^{z_2}
\left|\frac{f'(z)}{f(z)}\right|
|dz|
= s\int_{\theta_1}^{\theta_2}
\left|\frac{f'(se^{i\theta})}{f(se^{i\theta})}\right|d\theta.
\end{aligned}
\end{equation}
Now
\begin{equation}\label{p3xx}
\begin{aligned}
s\int_{\theta_1}^{\theta_2}
\left|\frac{f'(se^{i\theta})}{f(se^{i\theta})}\right| d\theta
&
=
s \int_{J_s} \left|\frac{f'(se^{i\theta})}{f(se^{i\theta})}\right|
d\theta + s \int_{[\theta_1,\theta_2]\setminus J_s}
\left|\frac{f'(se^{i\theta})}{f(se^{i\theta})}\right| d\theta
\\
&\leq
s \frac{T(r,f) [\log T(r,f)]^7}{r}\length J_s +
2\pi s\frac{T(r,f)^{1-\delta}}{r}  \\
&\leq
2 T(r,f) [\log T(r,f)]^7\length J_s +
4\pi\, T(r,f)^{1-\delta},
\end{aligned}
\end{equation}
where we used $s\leq R_3(r)\leq 2r$ in the last inequality.
Since $T(r,f)\leq \log M(r,f)$
we have
\[
4\pi\, T(r,f)^{1-\delta}\leq\frac14 \log M(r,f)
\]
for large $r$
and this, together with~\eqref{p3ww} and~\eqref{p3xx},
yields
\[
\frac14 \log M(r,f)\leq
2 T(r,f) [\log T(r,f)]^7\length J_s.
\]
We deduce that
\[
\length J_s
\geq\frac{1}{8 [\log T(r,f)]^7}
\]
and thus obtain~\eqref{3n1}.

Now we consider the case that~\eqref{p3h} does not hold.
Then~\eqref{3n} holds for all $\theta\in[0,2\pi]$.
Let $J_s$ be the subset of all $\theta\in[0,2\pi]$
for which~\eqref{3o} holds.
By the argument principle, we have
\[
n(s,0)=\frac{1}{2\pi i}\int_{|z|=s} \frac{f'(z)}{f(z)} dz
\leq
 \frac{s}{2\pi} \int_0^{2\pi} \left|\frac{f'(se^{i\theta})}{f(se^{i\theta})}\right|
d\theta.
\]
The same argument as in~\eqref{p3xx} now yields
\begin{equation}\label{p3l}
2\pi \, n(s,0)
\leq
2 T(r,f) [\log T(r,f)]^7\length J_s +
4\pi\, T(r,f)^{1-\delta}.
\end{equation}
Since we are assuming that~\eqref{p3h} does not hold,
we have $|f(re^{i\theta})|\geq 1$ for $0\leq\theta\leq 2\pi$ 
and thus $m(r,1/f)=0$, where $m(r,\cdot)$ denotes the Nevanlinna
proximity function.
Nevanlinna's first fundamental theorem, together with the
assumption that $f(0)=1$, yields
\[
T(r,f)= N(r,0)\leq N(1,0)+ n(r,0)\log r 
\]
and thus
\begin{equation}\label{p3m}
2\pi\, n(s,0)\geq 2\pi\, n(r,0)\geq \frac{T(r,f)}{\log r}
\geq 2 T(r,f)^{1-\delta/2}
\end{equation}
for large $r$ by~\eqref{2b}.
Combining~\eqref{p3l} and~\eqref{p3m} we obtain
$$
T(r,f)^{1-\delta/2} \leq 2 T(r,f) [\log T(r,f)]^7\length J_s ,
$$
from which~\eqref{3n1} easily follows.
\end{proof}
For small $\delta>0$ and large $r\notin E$ we consider the set
\[
\begin{aligned}
A(r):=\bigg\{
z
\in\C
:
\
& R_1(r)\leq |z|\leq R_3(r),\;\\
&|f(z)|\geq \sqrt{M(r,f)},\;
\left|\frac{f'(z)}{f(z)}\right|
\geq \frac{T(r,f)^{1-\delta}}{r}
\bigg\}.
\end{aligned}
\]
With the notation of Lemmas~\ref{lemma31} and ~\ref{lemma32}
we have
\[
A(r)\supset
\left\{se^{i\theta}: s\in F_r,\;\theta\in J_s\right\}
\]
and thus we can deduce from these lemmas that
\begin{equation}\label{4b}
\begin{aligned}
\area A(r)
&\geq
\length (F_r) 
\inf_{s\in F_r}( s\length(J_s))\\
&\geq
\frac{r^2}{[\log T(r,f)]^2 T(r,f)^\delta}
\geq
\frac{2r^2}{ T(r,f)^{2\delta}}
\end{aligned}
\end{equation}
for large $r$.
We put
$$
\rho(r):=\frac{r}{ T(r,f)^{1-2\delta}}
$$
and
\[
B(r):=A(r)\setminus \bigcup_{j=1}^l D\left(u_j,s_j+\rho(r)\right),
\]
where $l,u_1,\dots,u_l,s_1,\dots,s_l$ are chosen according to
Lemma~\ref{lemma31}, taking $\eta:=3\delta$ there.

We use the notation $\ann(r,R):=\{z\in\C:r<|z|<R\}$ for an annulus
with radii $r$ and~$R$.
 Given a constant $M>1$ we have $M\rho(r)\leq r/[\log
T(r,f)]^2$ for large $r$ and thus
\begin{equation}\label{p4xy}
D(b,M\rho(r))\subset \ann(r,R_4(r)) \quad\text{for} \ b\in A(r).
\end{equation}
Hence 
\[
D(b,\rho(r))\subset \ann(r,R_4(r))\setminus
\bigcup_{j=1}^l D(u_j,s_j)
\]
and thus
\begin{equation}\label{3e1}
\left|\frac{f'(z)}{f(z)}\right|
\leq \frac{T(r,f)^{1+3\delta}}{r}
\quad\text{if}\
b\in B(r) \ \text{and}\
z\in D(b,\rho(r))
\end{equation}
by~\eqref{3e}.
In particular, $f$ has no zeros in $D(b,\rho(r))$.

We want to show that the area of $B(r)$ is not much smaller
than that of $A(r)$.
In order to do so we note that~\eqref{3d1}, \eqref{3j1} and
the definition of $\rho(r)$
yield
\[
\begin{aligned}
\area \left(\bigcup_{j=1}^l D\left(u_j,s_j+\rho(r)\right)\right)
&=\pi\sum_{j=1}^l\left(s_j+\rho(r)\right)^2\\
&\leq 2\pi\sum_{j=1}^l \left(s_j^2+\rho(r)^2\right)\\
&\leq 2\pi \frac{r^2}{ T(r,f)^{3\delta}}
+2\pi \rho(r)^2 n(R_5(r),0)\\
&\leq 2\pi\frac{r^2}{ T(r,f)^{3\delta}}
+4\pi e \frac{r^2[\log T(r,f)]^2}{T(r,f)^{1-4\delta}}\\
&\leq \frac{r^2}{ T(r,f)^{2\delta}}
\end{aligned}
\]
for large~$r$, provided $\delta<1/6$.
 Combining this with~\eqref{4b} we obtain
\begin{equation}\label{4e}
\area B(r)
\geq \area A(r)-
\area \left(\bigcup_{j=1}^l D\left(u_j,s_j+\rho(r)\right)\right)
\geq \frac{r^2}{ T(r,f)^{2\delta}}
\end{equation}
for large~$r$.

Now let $m(r)$ be the maximal number of pairwise disjoint
disks of radius $\rho(r)$ whose centers are in $B(r)$.
Then $B(r)$ is contained in a union of $m(r)$ disks
of radius $2\rho(r)$ and thus
\[
\area B(r)\leq 4\pi m(r) \rho(r)^2=4\pi m(r)\frac{r^2}{T(r,f)^{2-4\delta}}.
\]
for large~$r$. Together with~\eqref{4e} we obtain
\begin{equation}\label{4g}
m(r)\geq \frac{1}{4\pi} T(r,f)^{2-6\delta}
\geq T(r,f)^{2-7\delta}.
\end{equation}

Recall that if $b\in B(r)$, for some large $r\notin E$, then
$f$ has no zeros in $D(b,\rho(r))$.
Thus we can define a branch $\phi$
of the logarithm of $f$ in $D(b,\rho(r))$;
that is, there exists a holomorphic function
$\phi:D(b,\rho(r))\to \C$ such that $\exp \phi(z)=f(z)$ for
$z\in D(b,\rho(r))$.
Of course, the other branches of the logarithm of $f$ are then given
by $z\mapsto \phi(z)+2\pi i n$ where $n\in \Z$.

For $a\in\R$ we will consider the domain
\[
Q(a):=\{z\in\C: |\re z-a|< 1,|\im z|< 2\pi\}
\]
\begin{lemma}\label{lemma33}
Let $b\in B(r)$, with $r\notin E$ sufficiently large.
Then there exists a branch $\phi$ of the logarithm of $f$
defined in $D(b,\rho(r))$ and a subdomain $U$ of $D(b,\rho(r))$
such that
$\phi$ maps $U$ bijectively onto $Q(\log|f(b)|)$.
\end{lemma}
\begin{proof}
Let $\phi_0$ be a fixed branch of the logarithm of $f$ defined in
$D(b,\rho(r))$. Define $h:\D\to \C$,
$h(z)=\phi_0(b+\rho(r)z)-\phi_0(b)$. Then $h(0)=0$ and
\[
|h'(0)|=\rho(r)|\phi_0'(b)|=\rho(r)\left|\frac{f'(b)}{f(b)}\right|
\geq \rho(r)\frac{T(r,f)^{1-\delta}}{r} = T(r,f)^{\delta}
\]
by the definition of $\rho(r)$ and
$A(r)$ and since $B(r)\subset A(r)$.
For $k\in\{1,2,3\}$ we put
$D_k:=\{z\in\C:|\re z|<1,|\im z -12k\pi |<4\pi\}$.
Lemma~\ref{ahlfors} now implies that if $r\notin E$ is large
enough, then there exists a subdomain $V$ of $\D$ and
$k\in\{1,2,3\}$ such that $h$ maps $V$ bijectively
onto $D_k$.
This implies that $D(b,\rho(r))$ contains a subdomain $W$
which is mapped bijectively onto
$$\{z\in\C: \left|\re z-\log|f(b)|\right|< 1,
|\im z-\im \phi_0(b)-12\pi k|< 4\pi\}$$ by $\phi_0$.
 Choosing $n\in \Z$ 
such that $|2\pi n-\im \phi_0(b) -12k\pi |\leq \pi$ we find
that there exists a domain $U\subset W\subset D(b,\rho(r))$ which is
mapped bijectively onto
$$\{z\in\C: \left|\re z-\log|f(b)|\right|< 1,|\im z-2\pi n|< 2\pi\}$$
by $\phi_0$.
The branch $\phi$ of the logarithm given by $\phi(z)=\phi_0(z)-2\pi i n$
now has the required property.
\end{proof}
We note that if $U$ is as in Lemma~\ref{lemma33}, then
\begin{equation}\label{4f}
f(U)=\exp Q(\log|f(b)|)=\ann\left(\frac{1}{e}|f(b)|,e|f(b)|\right).
\end{equation}

Now we fix some large $r_0\notin E$, put $\rho_0:=\rho(r_0)$
and choose $b_0\in B(r_0)$.
Lemma~\ref{lemma33} yields a branch $\phi_0$ of the logarithm
of $f$ and a domain $U_0\subset D(b_0,\rho_0)$ which is mapped
bijectively onto $Q(\log|f(b_0)|)$ by $\phi_0$.
Since $|f(b_0)|\geq \sqrt{M(r_0,f)}$
we also have 
$|f(b_0)|\geq 2r_0$
and the interval $\left[|f(b_0)|,2|f(b_0)|\right]$ contains
a point $r_1$ which does not belong to~$E$,
provided $r_0$ is chosen large enough.
Note that $r_1\geq 2r_0$.
We now choose $b_1\in B(r_1)$ and
with $\rho_1:=\rho(r_1)$
we find a domain $U_1\subset D(b_1,\rho_1)$
which is mapped
bijectively onto $Q(\log|f(b_1)|)$ by a branch $\phi_1$ of the logarithm
of $f$.

Inductively we thus obtain sequences $(r_k)$, $(\rho_k)$, $(b_k)$,
$(U_k)$ and $(\phi_k)$ satisfying
 $$r_{k}\in
\left[\,\left|f(b_{k-1})|,2|f(b_{k-1})\right|\,\right]\setminus E
 \subset [2r_{k-1},\infty)\setminus E,$$
 $\rho_k=\rho(r_k)$ and $b_k\in B(r_k)$
such that $U_k$ is a subdomain of $D(b_k,\rho_k)$  and $\phi_k$ is a
branch of the logarithm of $f$ which has the property that
 $\phi_k:U_k\to Q(\log|f(b_k)|)$
is bijective.

For large $r_0$ we have
$R_4(r_k)\leq 5 r_k/4 \leq 5|f(b_k)|/2$
for all $k$ and thus
\begin{equation} \label{A54}
\begin{aligned} 
D(b_k,M \rho_k)
&\subset 
\ann\left(r_k,R_4(r_k)\right)
\\ &
\subset
\ann\left(|f(b_{k-1})|,\frac52 |f(b_{k-1})|\right)
\subset 
\exp Q(\log|f(b_{k-1})|)
\end{aligned}
\end{equation}
by~\eqref{p4xy} and~\eqref{4f}.
Hence there exists a branch $L$ of the logarithm
which maps $D(b_k,M \rho_k)$ into $Q(\log |f(b_{k-1})|)$. 
Then $\psi_k:= \phi_{k-1}^{-1}\circ L$ 
is a branch of the inverse function of $f$ which maps
$D(b_k,M\rho_k)$ into $U_{k-1}$.

We put
\[
V_k:=\left(\psi_1\circ\psi_2\circ\ldots\circ\psi_k\right)\left(
\overline{D}(b_k,\rho_k)\right).
\]
Then $V_k$ is compact and $V_{k+1}\subset V_k$. Thus
$\bigcap_{k=1}^\infty V_k\neq\emptyset$. We will show that this
intersection contains only one point.

 In order to do so we note that since
$\psi_k:D(b_k,M\rho_k)\to U_{k-1}$ is univalent and $U_{k-1}\subset
D(b_{k-1},\rho_{k-1}) \subset D(\psi_k(b_k),2\rho_{k-1})$,
Koebe's one quarter
theorem implies that $2\rho_{k-1}\geq M\rho_k|\psi_k'(b_k)|/4$ and
thus $|\psi_k'(b_k)|\leq 8\rho_{k-1}/(M\rho_k)$. 
The Koebe distortion theorem, applied with $\lambda=1/M$,
now yields
$$
\sup_{z\in \overline{D}(b_k,\rho_k)}|\psi_k'(z)| 
\leq 
\frac{1+\lambda}{(1-\lambda)^3} |\psi_k'(b_k)|
=\frac{M^2(M+1)}{(M-1)^3} |\psi_k'(b_k)|  
\leq \frac{8M(M+1)}{(M-1)^3} \frac{\rho_{k-1}}{\rho_k}.
$$
Choosing $M=20$ we obtain
$$
\sup_{z\in \overline{D}(b_k,\rho_k)}|\psi_k'(z)| 
\leq \frac12 \frac{\rho_{k-1}}{\rho_k}.
$$
If $K\subset \overline{D}(b_k,\rho_k)$ is compact, we thus have
$$
\diam \psi_k(K)
\leq \frac12 \frac{\rho_{k-1}}{\rho_k}
\diam K,
$$
where $\diam K$ denotes the diameter of~$K$.
Hence
$$
\frac{\diam \psi_k(K)}{\rho_{k-1}}\leq \frac12 \frac{\diam
K}{\rho_k}.
$$
Inductively we obtain
$$
 \frac{\diam V_k}{\rho_0}
 =\frac{\diam 
\left( \psi_1\circ\psi_2\circ\ldots\circ\psi_k\right)
\left(\overline{D}(b_k,\rho_k)\right)}{\rho_{0}}
\leq\frac{1}{2^k} \frac{\diam \overline{D}(b_k,\rho_k)}{\rho_k}
= \frac{1}{2^{k-1}}
$$
and thus
\begin{equation}\label{4k}
\lim_{k\to\infty} \diam V_k = 0
\end{equation}
so that
$$
\bigcap_{k=1}^\infty V_k=\{z_0\}
$$
 for some~$z_0$.

It follows from the definition of $V_k$ and~\eqref{p4xy} that
$$
f^k(V_k)=\overline{D}(b_k,\rho_k)\subset \ann\left(r_k,R_4(r_k)\right)
$$
and hence that $z_0\in I(f)$.
Moreover,
$$
\begin{aligned}
f^{k+1}(V_k) 
&= f\left(\overline{D}(b_k,\rho_k)\right)
\supset f(U_k) \\
&=
\ann\left(\frac{1}{e}|f(b_k)|, e|f(b_k)| \right) 
\supset 
\ann\left(\frac{1}{e}r_{k+1},\frac{e}{2} r_{k+1} \right)  .
\end{aligned}
$$
As $F(f)$ does not have multiply connected components,
$$
\ann\left(\frac{1}{e}r_{k+1},\frac{e}{2} r_{k+1} \right)  
\cap
J(f)\neq\emptyset
$$
 for large~$k$. Since $J(f)$ is completely invariant, we
conclude that $V_k$ intersects $J(f)$, and
since $J(f)$ is  closed, this yields that $z_0\in I(f)\cap
J(f)$.

In order to estimate the upper box dimension of $I(f)
\cap J(f)$, we note that in
the above construction of the sequences $(r_k)$, $(\rho_k)$,
$(b_k)$, $(U_k)$ and $(\phi_k)$ we have $m(r_k)$ choices
for the point $b_k\in B(r_k)$ such that the disks of radius
$\rho_k$ around these points are pairwise disjoint, with $m(r_k)$
satisfying~\eqref{4g}.
We fix $k\in\N$, choose 
$r_j$, $\rho_j$, $b_j$, $U_j$ and $\phi_j$ for 
$0\leq j\leq k-1$  as before
and denote by
$b_k^\nu$
 choices of $b_k$ with the
above property,
with $1\leq \nu\leq m_k:=m(r_k)$

In other words, we take
$r_k\in [|f(b_{k-1})|,2|f(b_{k-1})|]\setminus E$
and $\rho_k:=\rho(r_k)$ 
as before and choose $b_k^\nu\in B(r_k)$, where $1\leq \nu\leq m_k$,
such that 
\begin{equation}\label{pwd}
D(b_k^i,\rho_k)\cap D(b_k^j,\rho_k)=\emptyset
\quad\text{for}\ i\neq j.
\end{equation}
Then for $1\leq \nu\leq m_k$ there exists a
branch $\psi_k^\nu:D(b_k^\nu,2 \rho_k)\to U_{k-1}$ of the inverse 
function of $f$
that is of the form $\psi_k^\nu=\phi_{k-1}^{-1}\circ L_\nu$ for 
some branch  $L_\nu$ of the logarithm which maps $D(b_k^\nu,2 \rho_k)$
into $Q(\log|f(b_{k-1})|)$.
For $a\in \R$ we  put
\[
P(a):=
\left\{z\in\C:\, 0\leq \re z-a\leq  \log \frac52,\,
|\im z|\leq  \frac32 \pi\right\} .
\]
In view of~\eqref{A54}
we may actually assume that 
$L_\nu$ maps $D(b_k^\nu,2 \rho_k)$ into 
the compact subset 
$P(\log|f(b_{k-1})|)$
of $Q(\log|f(b_{k-1})|)$.
Since $\phi_{k-1}^{-1}$ is univalent in $Q(\log|f(b_{k-1})|)$
we thus conclude that
there exists an absolute constant $\alpha>1$ such that
\begin{equation}\label{alphaphi}
\frac{1}{\alpha}\leq 
\frac{|(\phi_{k-1}^{-1})'(\zeta)|}{|(\phi_{k-1}^{-1})'(z)|}
\leq \alpha
\quad\text{for}\ 
\zeta,z\in P(\log|f(b_{k-1})|).
\end{equation}
An explicit upper bound
 for $\alpha$ could be determined from the Koebe
distortion theorem, but we do not need  such an estimate.
Put 
\[
\Lambda_k^\nu:=\psi_1\circ\psi_2\circ \ldots\circ \psi_{k-1}\circ\psi_k^\nu
=\psi_1\circ\psi_2\circ \ldots\circ \psi_{k-1}\circ\phi_{k-1}^{-1}\circ L_\nu.
\]
Since $\psi_1\circ\psi_2\circ \ldots\circ \psi_{k-1}$ is univalent in
$D(b_{k-1},2\rho_{k-1})$, 
we deduce from~\eqref{alphaphi}  and the Koebe distortion theorem
that there exists $\beta>1$ such that
\begin{equation}\label{betaLambda}
\frac{1}{\beta}\leq 
\frac{|(\Lambda_k^\nu)'(b_k^\nu)|}{|(\Lambda_k^1)'(b_k^1)|}
\leq \beta
\quad\text{for}\ 
1\leq \nu\leq m_k.
\end{equation}

We put
\[
V_k^\nu:=\Lambda_k^\nu\left(\overline{D}(b_k^\nu,\rho_k)\right)
\quad\text{and}\quad
v_k^\nu:=\Lambda_k^\nu(b_k^\nu).
\]
As above we see that each $V_k^\nu$
contains a point of
$I(f)\cap J(f)$. 

Since $\Lambda_k^\nu$ is univalent in $D(b_k^\nu,2\rho_k)$ we
deduce fron the 
Koebe distortion theorem (with $\lambda=1/2$) that 
\begin{equation}\label{Vkl}
\overline{D}\left(v_k^\nu,\frac49 \sigma_k^\nu\right)
\subset 
V_k^\nu\subset \overline{D}\left(v_k^\nu,4 \sigma_k^\nu\right)
\end{equation}
where
\[
\sigma_k^\nu:=\left|(\Lambda_k^\nu)'(b_k^\nu)\right|\rho_k
=\frac{\rho_k}{\left|(f^k)'(v_k^\nu)\right|}.
\]
We put $\sigma_k:=\sigma_k^1$.
It follows from~\eqref{4k} and~\eqref{Vkl} that
$$\lim_{k\to\infty}\sigma_k=0.$$
Using~\eqref{betaLambda} 
and~\eqref{Vkl}
we obtain
\begin{equation}\label{sizeWkl}
\overline{D}\left(v_k^\nu,\frac{4}{9\beta} \sigma_k\right)
\subset 
V_k^\nu\subset \overline{D}\left(v_k^\nu,4\beta \sigma_k\right)
\end{equation}

Fix a square of sidelength $\sigma_k$
centered at a point~$c$ and denote by 
$N$ the cardinality of the set of all $\nu\in\{1,\dots,m_k\}$ for 
which $V_k^\nu$ intersects this square.
It follows from~\eqref{sizeWkl} that if 
if $V_k^\nu$ intersects  this square,
 then $V_k^\nu\subset D(c,(8\beta+1)\sigma_k)$.
On the other hand, \eqref{sizeWkl} also says
that $V_k^\nu$ contains a disk of 
radius $4\sigma_k/(9\beta)$. Since the 
$V_k^\nu$ have pairwise disjoint interior by~\eqref{pwd},
we obtain
\[
N\pi 
\left(\frac{4\sigma_k}{9\beta} \right)^2
\leq 
\pi 
\left( (8\beta+1)\sigma_k\right)^2.
\]
Thus 
$N\leq N_0:=\lfloor 81\beta^2(8\beta+1)^2/16\rfloor$.

We now put a grid of sidelength $\sigma_k$ over $U_0$.
Then each square of this grid can intersect at most 
$N_0$ of the $m_k$ domains $V_k^\nu$.
Recalling that each of the domains $V_k^\nu$ contains
a point of $I(f)\cap J(f)$
we see that at least $m_k/N_0$ squares of our grid 
intersect $I(f)\cap J(f)$.
We conclude that
\begin{equation}\label{dimmk}
\dimUB(U_0\cap I(f)\cap J(f))
\geq
\limsup_{k\to\infty}
\frac{\log (m_k/N_0)}{-\log \sigma_k}.
\end{equation}
By~\eqref{4g} we have 
\begin{equation}\label{boundmk}
m_k=m(r_k)\geq T(r_k,f)^{2-7\delta}.
\end{equation}
It remains to estimate~$\sigma_k$.
In order to 
do so
we note that
\begin{equation}\label{4l}
(f^k)'(v_k^1)=\prod_{j=0}^{k-1}f'(f^j(v_k^1)).
\end{equation}
Since
\[
f^j(v_k^1)\subset U_j\subset D(b_j,\rho_j)
\]
it follows from~\eqref{3e1}  that
\[
\left|\frac{f'(f^j(v_k^1))}{f(f^j(v_k^1))}\right|
\leq \frac{T(r_j,f)^{1+3\delta}}{r_j}.
\]
This yields
\begin{equation}\label{4m}
\left|
f'(f^j(v_k^1))\right|
\leq \frac{T(r_j,f)^{1+3\delta}}{r_j}\left|
f^{j+1}(v_k^1)\right|
\leq e \frac{T(r_j,f)^{1+3\delta}}{r_j} r_{j+1}.
\end{equation}
We deduce from~\eqref{4l} and ~\eqref{4m} that
\[
\left|(f^k)'(v_k^1)\right| \leq
 \prod_{j=0}^{k-1} 
e \frac{T(r_j,f)^{1+3\delta}}{r_j} r_{j+1}
=e^k  \frac{r_{k}}{r_0}
\left(\prod_{j=0}^{k-1} T(r_j,f)\right)^{1+3\delta}
\]
Using the definition of $\rho_k$ and $\sigma_k$ we thus have
\begin{equation}\label{4n}
\sigma_k\geq
\frac{r_0}{e^k T(r_k,f)^{1-2\delta} \left( \prod_{j=0}^{k-1}
T(r_j,f)\right)^{1+3\delta}}
\end{equation}
By construction, we have
\[
\log r_{j+1}\geq \log|f(b_j)|\geq \frac12\log M(r_j,f)
\]
for $0\leq j\leq k-1$.
Given a large positive number $q$, we deduce from~\eqref{1a} that
if $r_0$ is sufficiently large, then
\[
\log M(r_{j+1},f)\geq \left(2\log r_{j+1}\right)^{q+1}
\]
for $0\leq j\leq k-1$.
Combining the last two estimates with~\eqref{3b1}
we find that
\[
T(r_j,f)\leq \log M(r_j,f)\leq 2\log r_{j+1}
\leq (\log M(r_{j+1},f))^{1/(q+1)}
\leq T(r_{j+1},f)^{1/q}
\]
for large $r_0$.
We conclude that
\begin{equation}\label{4o}
\prod_{j=0}^{k-1} T(r_j,f)
\leq T(r_k,f)^{\tau},
\end{equation}
where
\[
\tau:=\sum_{j=1}^k\left(\frac{1}{q}\right)^j
\leq \sum_{j=1}^\infty \left(\frac{1}{q}\right)^j
=\frac{1}{q-1}.
\]
For large $q$ we have $\tau(1+3\delta)\leq\delta$ and
thus
\[
\left( \prod_{j=0}^{k-1}
T(r_j,f)\right)^{1+3\delta}
\leq T(r_k,f)^\delta.
\]
We can also deduce from~\eqref{4o} that
\[
T(r_k,f)^\delta
\geq e^k
\]
if $r_0$ is chosen large enough. Combining the last two estimates
with~\eqref{4n}, and assuming that $r_0\geq 1$, we conclude that
 $\sigma_k\geq 1/T(r_k,f)$.

Together with~\eqref{dimmk} and~\eqref{boundmk} we thus find
that
\[
\dimUB(U_0\cap I(f)\cap J(f))
\geq\limsup_{k\to\infty}\frac{(2-7\delta)\log T(r_k,f)-\log N_0}
{\log T(r_k,f)}
=2-7\delta.
\]
Since $\delta>0$ was arbitrary, we obtain~\eqref{1a2} for $U=U_0$.
We may assume that the closure of
$U_0$ does not intersect the exceptional set.
As mentioned in the introduction, Theorem~\ref{thm1} now follows.

\begin{remark}
Theorem~B has been extended to meromorphic functions with finitely
many poles~\cite{Rippon06} and in fact to meromorphic functions with a
logarithmic tract~\cite[Theorem~1.4]{BRS1}. It is conceivable that
our result admits similar extensions.
\end{remark}

\section{The minimum modulus of entire functions with multiply connected
Fatou components}
Zheng~\cite{Zheng06} proved that if  the Fatou set of
a transcendental entire
function $f$ has a multiply connected component~$U$,
then there exist sequences $(r_k)$ and $(R_k)$
satisfying
$\lim_{k\to\infty}r_k=\lim_{k\to\infty}R_k/r_k= \infty$
such that $\ann(r_k,R_k)\subset f^k(U)$
 for large $k$.
It is shown in~\cite{BRS} that
one can actually take $R_k=r_k^c$ for some $c>1$ and this is then
used to show that if $F(f)$ has a multiply connected component,
then there exists $C >0$ such that
\begin{equation}\label{5x}
\log L(r,f)\geq \left(1-\frac{C}{\log r}\right)\log M(r,f)
\end{equation}
on some unbounded sequence of $r$-values.

Using Zheng's result~\cite{Zheng06} instead of~\cite{BRS}  yields
the following proposition referred to in the introduction.
 We include its short proof for completeness.

\begin{proposition}\label{zhengmod}
Let $f$ be a transcendental entire function for which $F(f)$
has a multiply connected component. Then~\eqref{1y} holds.
\end{proposition}
\begin{proof}
Zheng~\cite[Corollary~1]{Zheng06} used hyperbolic geometry to prove
that the hypothesis of the proposition implies~\eqref{1z}.
 We will use the same idea and denote the hyperbolic
 distance of two points $a,b$ in a hyperbolic domain $V$ by
 $\lambda_V(a,b)$; see, e.g., \cite[Section 2.2]{McM94} for
the properties of the hyperbolic metric that are used.

Let $U$ be a multiply connected component of $F(f)$ and let $(r_k)$
and $(R_k)$ be as in Zheng's result mentioned above.
  Put $s_k=\sqrt{R_kr_k}$ and $U_k=f^k(U)$.
Choose $|a_k|=|b_k|=s_k$ such that
$|f(a_k)|=L(s_k,f)$ and $|f(b_k)|=M(s_k,f)$. Since
$R_k/r_k\to\infty$ we easily see that $\lambda_{U_k}(a_k,b_k)\leq
\lambda_{\ann(r_k,R_k)}(a_k,b_k)\to 0$. Thus
$\lambda_{U_{k+1}}(f(a_k),f(b_k))\to 0$. Since $0,1\notin U_{k+1}$
for large $k$
 we obtain
\begin{equation}\label{5a}
\lambda_{\C\setminus\{0,1\}}(f(a_k),f(b_k))\to 0
\end{equation}
as $k\to\infty$. As the
 density $\rho_{\C\setminus\{0,1\}}(z)$ of the hyperbolic metric in
$\C\setminus\{0,1\}$ satisfies $\rho_{\C\setminus\{0,1\}}(z)\geq
 c/(|z|\log|z|)$ for some $c>0$ and large~$|z|$,
 we obtain
\begin{equation}\label{5b}
\lambda_{\C\setminus\{0,1\}}(f(a_k),f(b_k))
\geq c\int_{|f(a_k)|}^{|f(b_k)|}\frac{dt}{t\log t}
 =c\log\frac{\log M(s_k,f)}{\log L(s_k,f)}.
\end{equation}
Now~\eqref{1y} follows from~\eqref{5a} and~\eqref{5b}.
\end{proof}

\begin{remark}
An alternative way to deduce Proposition~\ref{zhengmod} from Zheng's
result is via Harnack's inequality~\cite[p.~14]{Ransford}.
 This method was used by Hinkkanen~\cite[Lemma~2]{Hinkkanen94}
 and Rippon and Stallard~\cite[Lemma~5]{RS-slowescape},
 and it is also used in~\cite{BRS}.

With the notation as in the above proof, put $t_k=
\log s_k=\log \sqrt{R_k/r_k}$ and define
$$u_k:\{z\in\C: |\re z|<t_k\}\to\R, \quad u_k(z)=\log|f(s_ke^z)|.
$$
We may assume that $|f(z)|>1$ for $z\in \ann(r_k,R_k)\subset f^k(U)$
so that $u_k$ is a positive harmonic function.
 Choose $y_1,y_2\in\R$ with
$|y_1-y_2|\leq\pi$ such that $u(iy_1)=\log L(s_k,f)$ and
$u(iy_2)=\log L(s_k,f)$.
 By Harnack's inequality we have
$$
u(iy_2)\leq \frac{t_k+\pi}{t_k-\pi} u(iy_1)=(1+o(1))u(iy_1)
$$
as $k\to\infty$, and~\eqref{1y} follows.
\end{remark}
\begin{remark}\label{rem2}
It follows from a result of  Fenton (\cite{Fenton80}, see
also~\cite{Chyz}) that if
\begin{equation}\label{6a}
p:=\liminf_{r\to\infty}\frac{\log\log M(r,f)}{\log\log r}<\infty
\end{equation}
and $\eps>0$, then
\begin{equation}\label{6x}
\log L(r,f)\geq \log M(r,f) -  (\log r)^{p-2+\eps}
\end{equation}
on some sequence of $r$-values tending to $\infty$. Hence
\begin{equation}\label{6c}
\log L(r,f)\geq \log M(r,f) - \frac{\log M(r,f)}{(\log r)^{2-2\eps}}
\end{equation}
on such a sequence. Choosing $\eps<1/2$ we see that if
\begin{equation}\label{6b}
\lim_{r\to\infty}\left(1-\frac{\log L(r,f)}{\log M(r,f)}\right)\log
r= \infty,
\end{equation}
 then~\eqref{6c} and hence~\eqref{6a} cannot hold and thus~\eqref{1a} holds.
The result of~\cite{BRS} quoted before Proposition~\ref{zhengmod}
shows that if~\eqref{6b} holds, then $F(f)$ has no multiply
connected components.
Thus we obtain the following consequence of Theorem~\ref{thm1}.
\begin{cor}\label{cor3}
Let $f$ be a transcendental entire function satisfying~\eqref{6b}.
Then  $\dimP(I(f)\cap J(f))=2$.
\end{cor}
Finally we mention that it was actually shown
in~\cite{BRS} and~\cite{Fenton80} that~\eqref{5x}
and~\eqref{6x} hold on sets or $r$-values of a
certain size. This could be used to further strengthen
the statement of Corollary~\ref{cor3}.
\end{remark}

\end{document}